\documentclass[12pt]{amsart}
\pdfoutput=1
\usepackage{amssymb}
\usepackage{amsmath}
\usepackage{amsfonts}
\usepackage[usenames]{color}
\usepackage{graphicx}
\usepackage{array}
\usepackage{psfrag}
\usepackage{color}
\usepackage{ulem}
\usepackage{hyperref}

\makeatletter
 
 \@addtoreset{equation}{section}
\makeatother

\newtheorem{theorem}{Theorem}
\newtheorem{lemma}[theorem]{Lemma}

\theoremstyle{definition}
\newtheorem{definition}[theorem]{Definition}

\theoremstyle{remark}

\numberwithin{equation}{section}

\newcommand{\g}{\Gamma}

\newcommand{\bibtitle}[1]{\textit{#1}}

\begin{document}

\title{Escher squares and lattice links}

\author{
Ramin Naimi,
Andrei Pavelescu,
Elena Pavelescu
}

\address{
Department of Mathematics, Occidental College, Los Angeles, CA 90041, USA. 
}
\address{
Department of Mathematics, University of South Alabama, Mobile, AL  36688, USA.
}

\date{\today}


\maketitle

\begin{abstract}
We give a shorter and simpler proof of the result of \cite{AB},
which gives a necessary and sufficient condition
for when a lattice diagram is the projection of a lattice link.
\end{abstract}
\vspace{0.1in}

\section{Introduction}
 For $n= 2,3$, let $G_n$ be the infinite graph with vertex set $\mathbb{Z}^n\subset \mathbb{R}^n$ such that for each $v, w \in \mathbb{Z}^n$, $vw$ is an edge of $G_n$ if and only if it has length 1.
A \textit{lattice link} is a link $L \subset G_3$ with one or more components.
Much of the research on lattice links focuses on finding lattice stick numbers, the minimal number of line segments (possibly containing more than one edge) necessary to construct a link type as a lattice link \cite {ACCJSZ, Di, HNO, HO, JP}. 
One of the first results was that of Diao \cite{Di}, who proved that  the lattice stick number of the trefoil is 12.
Adams et al. \cite{ACCJSZ} found lattice stick numbers for various knots and links, including all $(p, p + 1)-$torus knots.
Hong, No and Oh \cite{HNO} found  all links with more than one component whose lattice stick numbers are at most 14. 
Lattice links have played an important role in simulating various circular molecules \cite{TROQ}.
One way to describe a link is through a diagram, a  projection in two dimensions with additional crossing information.
In this article, we characterize those diagrams which represent projections of lattice links. 

A \textit{lattice diagram} is a subgraph $D\subset G_2$ such that every vertex of $D$ has degree 2 or 4, and every vertex of degree 4 is  endowed with over/under crossing information as in Figure \ref{Fig-crossings}.
\begin{figure}[ht]
 \centering
 \includegraphics[width=70mm]{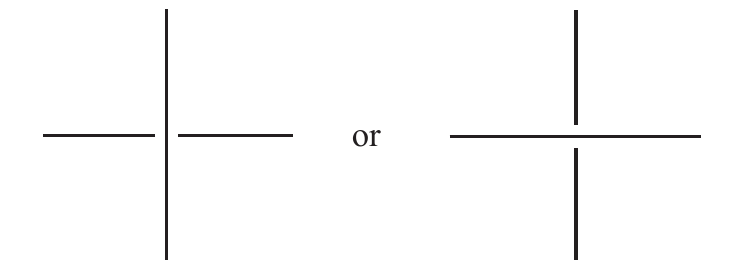}
\caption{Diagram crossings.}
\label{Fig-crossings}
\end{figure}

\begin{definition}
\label{realizable}
Let $\pi : \mathbb{R}^3 \rightarrow \mathbb{R}^2$ be given by $\pi(x,y,z)=(x,y)$. We say a lattice diagram $D$ is \textit{realizable} if there is a lattice link $L$ such that
\begin{itemize}
\item[(1)] $\pi(L)=D$;
\item[(2)] If $p$ belongs to the interior of an edge of $D$, then $\pi^{-1}(p)\cap L$ is a single point;
\item[(3)] If $v$ is a degree 2 vertex of $D$, then $\pi^{-1}(v)\cap L$ is connected;
\item[(4)] If $v$ is a crossing, i.e. a degree 4 vertex of $D$, then $\pi^{-1}(v)\cap L$ has exactly two connected components, and the component whose points have greater $z$-coordinates than those of the points of the other component is connected to the two edges of $L$ whose projections in $D$ are labelled as overstrands at $v$. \vspace{-0.2in}
\end{itemize}
\end{definition}

\medskip

An \textit{Escher square} in a lattice diagram is a 4-cycle $v_1 v_2 v_3 v_4$ such that, up to reversing all four crossings, every edge $v_{i} v_{i+1}$ is an understrand at $v_i$ and an overstrand at $v_{i+1}$  (where $v_{4+1} := v_1$).
Thus, an Escher square looks either as in the center of Figure~\ref{Fig-GraphTrefoil-Escher}(b), or its mirror image, as in Figure~\ref{Fig-Lemma4cycle}(c).
We named this configuration an Escher square as a reference to M.C. Escher's impossible stair case \cite{E}.

Allardice and Bloch \cite{AB} showed that
a lattice diagram is realizable if and only if it does not  contain an Escher square.
Here we give a shorter and simpler proof of their result.

Given a lattice diagram $D$, we construct an associated digraph as follows. 
Each edge of $D$ is a vertex of $\g$. 
If $e$ and $f$ are edges of $D$ which meet at a crossing with $e$ an understrand and $f$ an overstrand, 
then $\g$ contains an edge directed from vertex $e$ to vertex $f$,
which we denote as  $ef$.
Figure \ref{Fig-GraphTrefoil-Escher} shows two lattice diagrams and their associated digraphs.

\begin{figure}[ht]
 \centering
 \includegraphics[width=110mm]{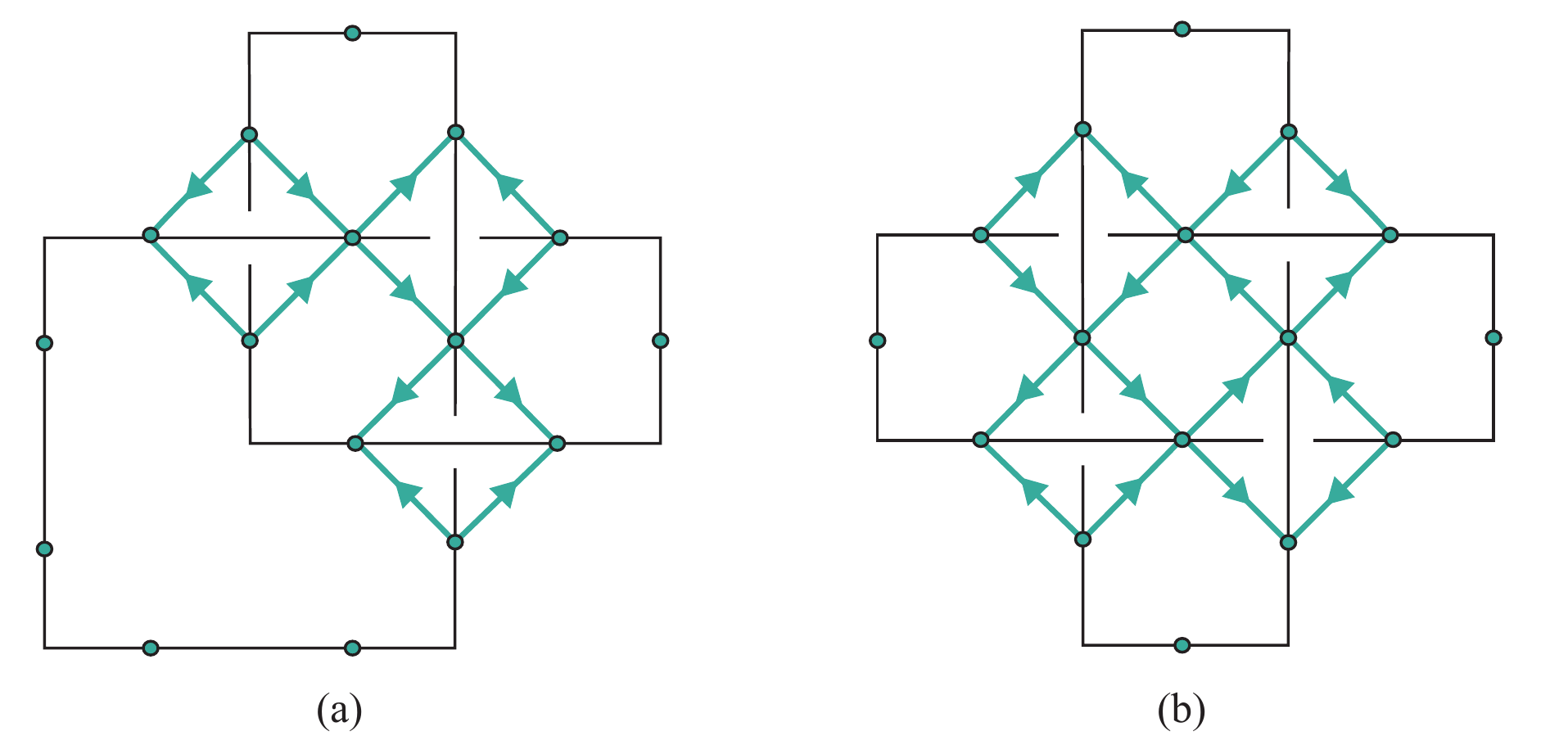}
\caption{Lattice diagrams and associated digraphs.}
\label{Fig-GraphTrefoil-Escher}
\end{figure}

An $n$-cycle in $\g$ is a sequence $e_1e_2...e_n$ of distinct vertices of $\g$ such that $e_ne_1$ and $e_ie_{i+1}$ are edges, for $1\le i \le n-1$.
A \textit{height function} on $\g$ is an integer-value function $h$ on the vertices of $\g$ such that for each edge $ef$, $h(e)<h(f)$.

\section{Main Result}

\begin{lemma} A lattice diagram $D$ is realizable if and only if  its associated digraph $\g$ admits a height function. 
\label{construction}
\end{lemma}

\begin{proof}
Let $\g$ be the digraph associated with a lattice diagram $D$. 
Suppose $D$ is realizable, with $\pi(L)=D$ as in Definition \ref{realizable}. 
For each edge $e$ of $D$, let $h(e)$ be the $z$-coordinate of a point in $L$ whose image under $\pi$ is in the interior of $e$. 
Then, by part 4 of Definition \ref{realizable}, for each edge $ef$ of $\g$, $h(e)<h(f)$. Hence $h$ is a height function on $\g$.

To show the converse, suppose there exists a height function $g$ on $\g$. 
For each edge $e$ of $D$, let $\tilde{e}$ be the edge in $G_3$ such that $\pi(\tilde{e})=e$ and such that $\tilde{e}$ lies on the plane $z=g(e)$. 
Let $e, f$ be a pair of edges that meet at a vertex $v$ of $D$. If $\mathrm{deg}(v)=2$, 
then connect $\tilde{e}$ and $\tilde{f}$ by a path
 in $\pi^{-1}(v)$. 

If $\mathrm{deg}(v)=4$, and if $e$ and $f$ are either both overstrands at $v$, or both understrands at $v$, then connect $\tilde{e}$ and $\tilde{f}$ by a path in $\pi^{-1}(v)$. 
 The fact that $g$ is a height function implies that the path connecting the ``lifts'' of the overstrands at $v$ is disjoint from the path connecting the ``lifts'' of the two understrands at $v$. 
 Thus we have constructed an embedded link in $G_3$ which satisfies all the properties in Definition \ref{realizable}, showing that $D$ is realizable.
\end{proof}

\begin{lemma} A digraph $\g$ admits a height function if and only if it has no cycles.
\label{cyclenoheight}
\end{lemma}

\begin{proof}
Let $h$ be a height function on a digraph $\g$. 
Then $\g$ cannot have any cycle $v_1,...,v_n$ because we would have $h(v_1)<h(v_2)<....<h(v_n)<h(v_1)$.

Conversely, suppose $\g$ is a digraph that has no cycles. 
Let $S$ be the set of vertices $v\in \g$ such that no edge is oriented towards $v$.
Since $\g$ has no cycles, the set  $S$ is non-empty.
For each $v\in S$, let $g(v)=0$.
For each $v\notin S$, let $g(v)$ be the length of the longest path from a vertex in $S$ to $v.$ 
Notice that all paths are simple since $\g$ has no cycles. 
Then, for each edge $vw$ in $\g$, $g(v)<g(w)$ since if $v_0v_1....v_n$ is a path with $v_0$ in $S$, $v_n=v$, and $n=g(v)$, then $v_0v_1....v_nw$ is a path from $v_0$ to $w$, which implies $g(w)\ge g(v)+1$. So $g$ is a height function on $\g$.
\end{proof}

\begin{lemma} The digraph $\g$ associated with a lattice diagram $D$ has no cycles if and only if it has no 4-cycles. 
\label{nocycleno4}
\end{lemma}

\begin{proof}
We say two edges in $G_2$ are \textit{aligned} if they are opposite edges in some unit square in $G_2$.
 See Figure \ref{Fig-Lemma4cycle}(a).
\begin{figure}[htpb!]
\begin{center}
\begin{picture}(420, 168)
\put(0,0){\includegraphics[width=5.7in]{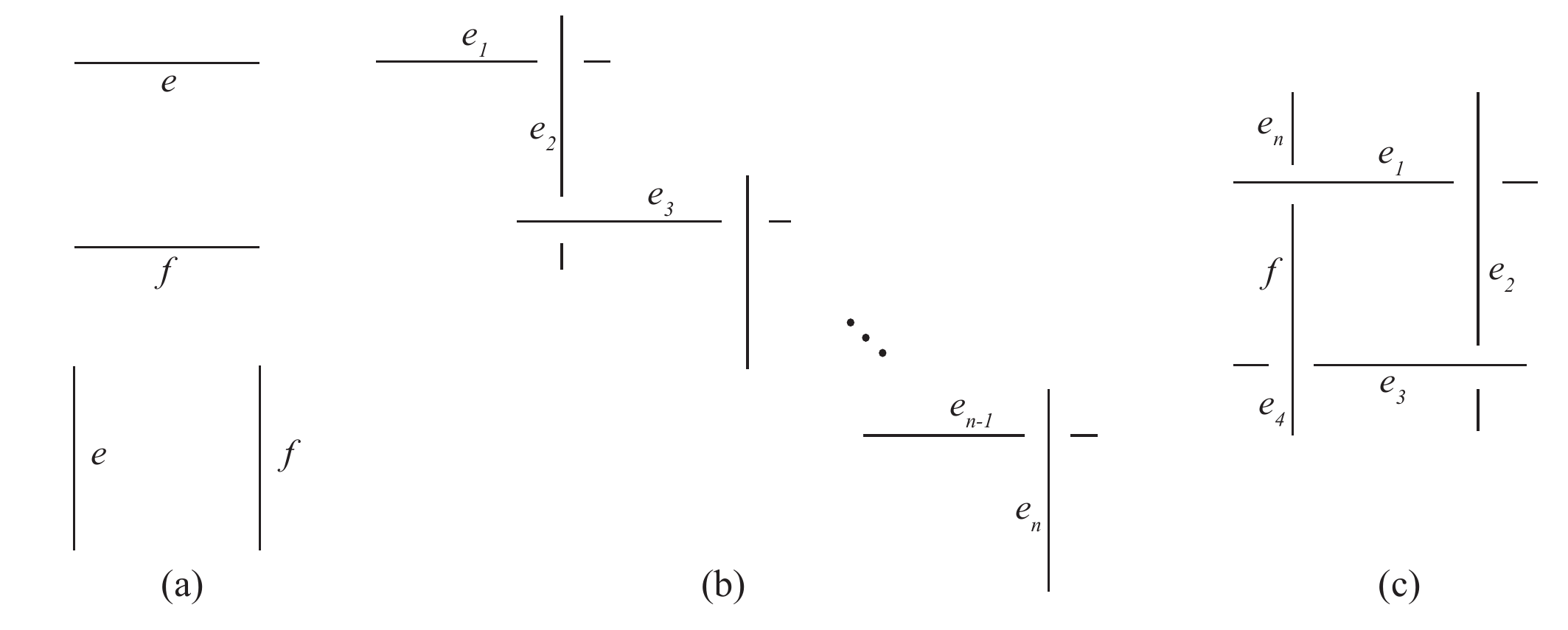}}
\end{picture}
\caption{(a) aligned edges; (b) edge configuration; (c) Escher square.}
\label{Fig-Lemma4cycle}
\end{center}
\end{figure}
Let $\g$ be the digraph associated with a lattice diagram $D$. 
Suppose $e_1e_2 ... e_n$ is a cycle in $\g$, $n>4$. 
Then for some $k$, $1\le k \le n$, $e_k$ and $e_{k+2}$ are aligned (where $e_{n+1}:=e_1$ and $e_{n+2}:=e_2$), since otherwise, up to reflection and rotation, the edges of $D$ corresponding to this cycle would have to look as in Figure~\ref{Fig-Lemma4cycle}(b); 
but this would contradict the assumption that $e_1$ and $e_n$ meet at a crossing. 
So, without loss of generality, suppose that $e_1$ and $e_3$ are aligned. 
Then, up to rotation and reflection, $D$ contains the picture in Figure~\ref{Fig-Lemma4cycle}(c).
In this figure, since $e_n$ meets $e_1$ as an understrand, so must $f$. 
Also, since $e_3$ meets $e_4$ as an understrand, $e_3$ must meet $f$ as an understrand as well. So $e_1e_2e_3f$  is a 4-cycle in $\g$.
\end{proof}

\begin{theorem} 
A lattice diagram is realizable if and only if it does not  contain an Escher square.  
\end{theorem}
\begin{proof}
This follows immediately from Lemmas \ref{construction}, \ref{cyclenoheight} and \ref{nocycleno4},
\end{proof}


\bibliographystyle{amsplain}

\end{document}